\documentclass[reqno,11pt]{amsart}
\usepackage{amsmath,amssymb}

\newtheorem{thm}{Theorem}[section]
\newtheorem{cor}[thm]{Corollary}
\newtheorem{lem}[thm]{Lemma}
\newtheorem{prop}[thm]{Proposition}

\theoremstyle{definition}

\newtheorem*{xrem}{Remark}

\numberwithin{equation}{section}

\newcommand{\Z}{\mathbb Z}
\newcommand{\N}{\mathbb N}
\newcommand{\R}{\mathbb R}

\newcommand{\Q}{\mathbb Q}

\newcommand{\DD}{\mathcal D}

\newcommand{\LL}{\mathcal L}
\newcommand{\RR}{\mathcal R}

\newcommand{\NN}{\mathcal N}
\renewcommand{\SS}{\mathcal S}

\newcommand{\PP}{\mathcal P}

\newcommand{\Li}{\operatorname{Li}}

\title[Statistical properties of continued fractions]{On certain statistical properties of continued fractions
with even and with odd partial quotients}

\author{Florin P. Boca}

\author{Joseph Vandehey}

\address{Department of Mathematics,
University of Illinois,
1409 W. Green Street,
Urbana, IL 61801}

\email{fboca@illinois.edu,\ 
vandehe2@illinois.edu}

\begin{document}


\begin{abstract}
We prove results concerning the joint limiting distribution of the renewal time of denominators and
consecutive digits of random irrational numbers in the case of continued fractions with even partial quotients,
with odd partial quotients, and for Nakada's $\alpha$-expansions.
\end{abstract}

\maketitle

\renewcommand{\thefootnote}{}



\footnote{FPB is a  member of the Institute of Mathematics ``Simion Stoilow" of the Romanian Academy,
Calea Grivi\c tei 21, RO-01072 Bucharest, Sector 1, Romania}

\renewcommand{\thefootnote}{\arabic{footnote}}
\setcounter{footnote}{0}

\section{Introduction}

Let $(a_n)$ (respectively, $(q_n)$) denote the sequence of digits (resp., denominators of the convergents) in the regular continued fraction (RCF)
expansion of an irrational number. For each $R>1$, consider the \emph{renewal time}
$n_R:=\min\{ n: q_n >R \}$, so that $q_{n_R-1} \leqslant R < q_{n_R}$. As a consequence of
their renewal-type theorem for the natural extension of the Gauss map associated with regular
continued fractions, Sinai and Ulcigrai \cite{SU} proved the existence of the joint limiting distribution
of $( q_{n_R-1}/R, R/q_{n_R}, a_{n_R-K},\ldots,a_{n_R+K})$, with $K$ a fixed
nonnegative integer, as $R\rightarrow\infty$. The classical Gauss-Kuz'min statistics give the probability of a random $x$ in $[0,1]$ having a prescribed string of digits in its continued fraction expansion at the $n$th position, for large $n$; the joint limiting distribution studied in \cite{SU,Ust} gives the probability of a random $x$ in $[0,1]$ having a prescribed string of digits in its continued fraction expansion at the first place where the denominator of the convergent is larger than $R$, for large $R$.  The joint limiting distribution may therefore be considered an analogue of Gauss-Kuz'min statistics. Employing an abstract characterization of denominators of
successive convergents in the regular continued fraction expansion ${\rm RCF}(x)$ of $x$, Ustinov succeeded
in explicitly computing this limiting distribution in the RCF case \cite{Ust}.

Sinai and Ulcigrai's result has been subsequently extended to the situation of continued fractions with even partial
quotients (ECF) by Cellarosi \cite{Cel}. The ECF limiting distribution was further used
in the renormalization of theta sums---that is, replacing the theta sum
$
\sum e^{\pi i \omega n^2}
$
with a theta sum of the type
$
\sum e^{-\pi i n'^2/\omega}
$
modulo a rescaling, rotation, and small error term---as the map $\omega \rightarrow -1/\omega$ modulo $2$ is closely related to the forward shift of even continued fractions.
This has led to some new results about the distribution of
normalized theta sums and geometrical properties of their associated curlicues \cite{Cel2,Si}.

This paper studies this type of limiting distributions in the case of three types of continued fractions:
ECF, OCF (continued fractions with odd partial quotients), and ${\rm NCF}_\alpha$ (the Nakada $\alpha$-expansions, which include NICF,
or continued fraction to the nearest integer, as a special case).
In the ECF case we provide a direct proof of the main result in \cite{Cel} while making the limiting
distribution explicit. The analogous problem is also solved in the OCF case, for which no ergodic theoretical approach is known at this time.
As in \cite{Ust}, the key tool is providing an abstract characterization for pairs of successive
convergents in ${\rm ECF}(x)$ and ${\rm OCF}(x)$, which may be of independent interest.
The OCF case is the most intricate, because the sequence of
denominators of successive convergents in ${\rm OCF}(x)$ is not necessarily increasing as in the RCF, ECF, or ${\rm NCF}_\alpha$ cases.
Finally we provide an explicit relation between the ${\rm NCF}_\alpha$ limiting joint distribution and the distribution computed in \cite{Ust}.

Concretely, for a given type of continued fraction expansion (ECF, OCF, or ${\rm NCF}_\alpha$), consider the \emph{renewal time}
\begin{equation*}
n_R=\min\{ n\in\N: q_n >R\} =\min\{ n\in\N : q_{n-1}\leqslant R <q_n \}, \quad R>1,
\end{equation*}
and the joint limiting distribution of
$(q_{n_R-1}/R, R/q_{n_R},$ $ \omega_{n_R-K} , \ldots , \omega_{n_R+K} )$ with $\omega_k=(a_k,e_k)$,
for fixed $K$, as $R\rightarrow\infty$.  Here again, $\omega_k$ denote the continued fraction digits and $q_n$ denote the denominators of the convergents for a given type of CF expansion (see Section \ref{section:basicproperties} for more details).

We will evaluate  the Lebesgue measure
$\LL^{E/O,\pm}_{x_1,x_2,x_3,x_4}(R)$ of the set of
 numbers $x\in\Omega:=[0,1]\setminus \mathbb{Q}$ for which there exist
successive convergents $P/Q,P'/Q'$ in ${\rm ECF}(x)$
(respectively in ${\rm OCF}(x)$) such that for given $x_1,x_2,x_3,x_4$ the following conditions are satisfied:
\begin{equation}\label{1.1}
\frac{Q}{R}\leqslant x_1,\quad \frac{R}{Q^\prime} \leqslant x_2,\quad
\frac{Q}{Q^\prime}\leqslant x_3 ,
\end{equation}
\begin{equation}\label{1.2}
0\leqslant \frac{Q^\prime x-P^\prime}{-Qx+P} \leqslant x_4\quad
\mbox{\rm respectively} \quad
-x_4 \leqslant \frac{Q^\prime x-P^\prime}{-Qx+P}\leqslant 0,
\end{equation}
depending on the choice of the $\pm$ sign.
In both ECF and OCF situations, we take $x_1,x_2,x_3,x_4$ $\in (0,1]$.\footnote{If any of the parameters equals $0$, then $\LL$ equals $0$ as well, so we ignore this degenerate case.} In the OCF case, the ratio $Q/Q'$ of successive denominators can in fact be any rational number in the interval $(0,G)$, but since in the definition of $n_R$ we are interested only in $Q\leqslant R< Q'$, we can restrict to $x_3 \leqslant 1$ in the definition of $\LL^{O,\pm}$. The golden ratios $G=(1+\sqrt{5})/2$ and $g=1/G=(-1+\sqrt{5})/2$ will be used often.

The terms $q_{n_R-1}/R$ and $R/q_{n_R}$ in the joint limiting distribution clearly relate to the parameters $x_1$ and $x_2$ in the function $\mathcal{L}$.
Likewise, the digits $\omega_k$ in the joint limiting distribution relate to the parameters $x_3$ and $x_4$ in $\mathcal{L}$ due to equalities \eqref{2.4} and \eqref{2.6} below.

The main result of this paper shows that $\LL^{E/O,\pm} (R)$ has an explicitly computable limiting distribution as $R\rightarrow\infty$.

\begin{thm}\label{Thm1}
The joint distributions $\LL^{E/O,\pm}_{x_1,x_2,x_3,x_4}(R)$ exist as $R\rightarrow \infty$ and
\begin{equation}\label{1.3}
\LL^{E,\pm}_{x_1,x_2,x_3,x_4} (R) = \frac{2F_\pm}{3\zeta(2)}+O_\varepsilon \big( R^{-1+\varepsilon}\big),
\end{equation}
\begin{align}\label{1.4}
\LL^{O,+}_{x_1,x_2,x_3,x_4} (R) & = \frac{F_+ -D_1}{\zeta(2)}+O_\varepsilon
\big( R^{-1/2+\varepsilon}\big),\\ \LL^{O,-}_{x_1,x_2,x_3,x_4} (R) & = \frac{F_- -D_2-D_3}{\zeta(2)}+O_\varepsilon
\big( R^{-1/2+\varepsilon}\big) ,\notag
\end{align}
where $F_\pm=F_\pm(x_1,x_2,x_3,x_4)$ and $D_i=D_i(x_1,x_2,x_3,x_4)$ are given by\footnote{In this paper the convention is that $\int_a^b =0$ when $a\geqslant b$.}
\begin{equation}\label{1.6}
F_\pm=\mp \begin{cases}
 \Li_2 (\mp x_1x_2x_4) & \mbox{\rm if $x_3 \geqslant x_1x_2,$} \\
\Li_2 (\mp x_3 x_4) - \log (1\pm x_3 x_4) \log\frac{x_1x_2}{x_3} & \mbox{\rm if $x_3 < x_1x_2,$}
\end{cases}
\end{equation}
\begin{equation}\label{1.7}
\begin{split}
D_2 & = F_- (x_1,x_2,x_3,x_4)-F_- (x_1,x_2,\min\{ x_3, g^2\} ,x_4), \\
D_1 & =\sum_{\ell \geqslant 1} I_\ell^+,\qquad D_3 =\sum_{\ell\geqslant 2} I_\ell^-,
\end{split}
\end{equation}
with
\begin{equation}\label{1.8}
I_\ell^\pm = \int^{A_\ell}_{1/x_2} dx
\int^{B_\ell (x)}_{x/(2\ell+g )}\frac{x_4\, dy}{y (y \pm x_4 x)},
\end{equation}
where
\begin{equation*}
A_\ell =(2\ell+g )x_1,\quad
B_\ell (x) =B_{\ell,x_2,x_3}(x)=\min\bigg\{ x_3 x,x_1, \frac{x}{2\ell},\frac{x-1}{2\ell -1}\bigg\}
.
\end{equation*}
\end{thm}
The integrals $I_\ell^\pm$ can be written explicitly as a combination of
logarithms and dilogarithms.

Kraaikamp's metric theory for $S$-expansions \cite{Kra} provides immediate characterizations of
pairs of successive convergents for such continued fractions, which are obtained from RCF only by
singularization (see the remark at the end of Section 3 for definition of singularization).
In the last section we show how to compute the joint limiting distribution associated as above with
Nakada's $\alpha$-expansions \cite{Nak} for $\frac{1}{2}\leqslant \alpha \leqslant 1$.
The cases $\alpha=1$ and $\alpha=\frac{1}{2}$ are best known, corresponding to the RCF and
NICF (continued fraction to the nearest integer). The latter was introduced by Minnigerode \cite{Min} and was also studied in \cite{Ada,Sch2,Wil}.
Our calculations show explicit connections with Ustinov's RCF distribution.

\section{Basic ECF and OCF properties}\label{section:basicproperties}

For each $x\in \Omega$, the ECF (respectively, OCF) expansion of $x$ is given by
\begin{equation}\label{2.1}
x=\cfrac{1}{a_1+\cfrac{e_1}{a_2+\cfrac{e_2}{a_3+\cfrac{e_3}{\ddots}}}} =[[(a_1,e_1),(a_2,e_2),(a_3,e_3),\ldots ]],
\end{equation}
where $e_n \in\{ \pm 1\}$ and all $a_n$'s are even positive integers (respectively, all $a_n$'s are odd positive integers
with $a_n+e_n \geqslant 2$). For more details see \cite{DHKM,Kra,Mas,Rie,Sch,Sch2}. As in \cite{DHKM,Mas}, consider the
``flipped" continued fraction map $T_D:[0,1]\rightarrow[0,1]$ for a subset $D$ of $[0,1]$, defined by $T_D(0)=0$, $T_D(1)=1$, and
\begin{equation*}
T_D(x) =\begin{cases}
\{ 1/x \} & \mbox{\rm if $x\in (0,1)\setminus D,$} \\
\displaystyle 1-\{ 1/x\}  & \mbox{\rm if $x\in D,$} \end{cases}
\end{equation*}
with auxiliary functions
\begin{equation*}
e_D (x) =\begin{cases} 1 & \mbox{\rm if $x\in [0,1]\setminus D,$} \\
-1 & \mbox{\rm if $x\in D,$} \end{cases}
\quad a_D (x) =\begin{cases} [ 1/x ] & \mbox{\rm if $x\in [0,1] \setminus D,$} \\
1+ [ 1/x] & \mbox{\rm if $x\in D.$} \end{cases}
\end{equation*}
Note that
\begin{equation*}
T_D(x)=e_D(x) \bigg( \frac{1}{x}-a_D(x)\bigg),\qquad \forall x\in (0,1).
\end{equation*}
Consider the sets
\begin{equation*}
D_O := \bigcup\limits_{n\in 2\N} \bigg[ \frac{1}{n+1},\frac{1}{n}\bigg),\qquad
D_E := [0,1) \setminus D_O=\bigcup\limits_{n\in 2\N -1} \bigg[ \frac{1}{n+1},\frac{1}{n}\bigg).
\end{equation*}
Denote $D=D_E$ in the ECF case, respectively $D=D_O$ in the OCF case.
In both ECF or OCF situations
the signs $e_n=e_n (x)$ and the digits $a_n=a_n(x)$ are given, for $x\in\Omega$, by
\begin{equation*}
e_0=1,\quad e_n =e_D(t_{n-1}), \quad
a_0=0,\quad a_n=a_D(t_{n-1}),
\end{equation*}
where $t_n=t_n(x)=T_D^n (x)$.  On the $D$-continued fraction expansion the iterates of the Gauss type map $T_D$ act as a shift map by
\begin{equation*}
T_D^n [[(a_1,e_1),(a_2,e_2),\ldots ]]= [[ (a_{n+1},e_{n+1}),(a_{n+2},e_{n+2}),\ldots ]],
\quad \forall n\in\N_0 .
\end{equation*}

The $D$-convergents $p_n/q_n$ are defined by
\begin{equation}\label{2.2}
\begin{cases}
p_{-1}=1,\  p_0=0,\  p_n=a_n p_{n-1}+e_{n-1}p_{n-2} , \\
q_{-1}=0,\  q_0=1,\  q_n =a_n q_{n-1}+e_{n-2} q_{n-2} ,
\end{cases}
\end{equation}
or in equivalent formulation
\begin{align}\label{2.3}
\left( \begin{matrix} p_{n-1} & p_n \\ q_{n-1} & q_n \end{matrix} \right) & =
\left( \begin{matrix} p_{n-2} & p_{n-1} \\ q_{n-2} & q_{n-1} \end{matrix} \right)
\left( \begin{matrix} 0 & e_{n-1} \\ 1 & a_n \end{matrix} \right) = \cdots \\ &
=\left( \begin{matrix} 0 & e_0 \\ 1 & a_1 \end{matrix}\right)
\left( \begin{matrix} 0 & e_1 \\ 1 & a_2 \end{matrix}\right) \cdots
\left( \begin{matrix} 0 & e_{n-1} \\ 1 & a_n \end{matrix}\right),
\quad \forall n\in \N.\notag
\end{align}
The following elementary fundamental relations are satisfied:
\begin{equation*}
\begin{split}
& p_{n-1} q_n -p_n q_{n-1}=(-1)^k e_0 e_1 \cdots e_{n-1} =:\delta_n,\\ &
\frac{p_{n-1}}{q_{n-1}}-\frac{p_n}{q_n} =\frac{\delta_n}{q_{n-1}q_n},
\quad \forall n\in\N_0 , \\
&
x=\frac{p_n+p_{n-1} e_n t_n}{q_n+q_{n-1} e_n t_n},\quad \forall n\in\N.
\end{split}
\end{equation*}
The latter equation is equivalent to
\begin{equation}\label{2.4}
e_n t_n =e_n T_D^n (x)
 =\frac{q_n x-p_n}{-q_{n-1} x+p_{n-1}},\quad
\forall n\in\N .
\end{equation}
Upon \eqref{2.4} we infer
\begin{equation}\label{2.5}
0< \bigg| \frac{q_n x-p_n}{-q_{n-1}x+p_{n-1}} \bigg| <1,\quad
\forall x\in\Omega,\ \forall n\in\N .
\end{equation}
It is well-known and plain to check for every continued fraction that if
$x$ is as in \eqref{2.1}, then
\begin{equation}\label{2.6}
\frac{q_{n-1}}{q_n}  =[[(a_n,e_{n-1}),(a_{n-1},e_{n-2}),\ldots,(a_2,e_1),(a_1,\ast)]],
\quad \forall n\in \N ,
\end{equation}
where $(a_1,\ast)$ means that the finite expansion terminates with $a_1$.

\section{Successive ECF and OCF convergents}

In $\operatorname{GL}_2(\Z)$ consider the matrices
\begin{equation*}
I=\left( \begin{matrix} 1 & 0 \\ 0 & 1 \end{matrix} \right),\quad
J=\left( \begin{matrix} 0 & 1 \\ 1 & 0 \end{matrix}\right) ,\quad
A=\left( \begin{matrix} 0 & 1 \\ 1 & 1 \end{matrix}\right), \quad
B=A^2 =\left( \begin{matrix} 1 & 1 \\ 1 & 0 \end{matrix}\right),
\end{equation*}
and denote their images in $\operatorname{SL}_2(\Z /2\Z)$ by $[I]$, $[J]$, $[A]$, $[B]$.
Clearly $\{ [I],[J]\}$ forms a subgroup on two elements of $\operatorname{SL}_2 (\Z / 2\Z)$
and $\{ [I],[A],[B]\}$ forms a subgroup on three elements of $\operatorname{SL}_2 (\Z / 2\Z)$.
Consider the sets
\begin{equation*}
\RR=\left\{ M=\left( \begin{matrix} P & P^\prime \\ Q & Q^\prime \end{matrix}\right)
:  0\leqslant P \leqslant Q,\ 1\leqslant P^\prime \leqslant Q^\prime,  \right\} ,
\end{equation*}
\begin{equation*}
\RR_E:=\{ M\in \RR: 1\leqslant Q\leqslant Q^\prime,\,
M \equiv I\ \mbox{\rm or}\ J \hspace{-5pt} \pmod{2} \},
\end{equation*}
\begin{equation*}
\RR_O:=\left\{ M\in\RR: \lambda_M  > g,\
M\equiv I, A, \ \mbox{\rm or} \ B \hspace{-5pt}\pmod{2}\right\} .
\end{equation*}
For $M\in \RR$ denote
\begin{equation}\label{3.1}
\lambda_M =\frac{Q^\prime}{Q},\qquad E_M(x)=\frac{Q^\prime x-P^\prime}{-Qx+P} ,\quad
x\notin \Q .
\end{equation}

\subsection{Successive convergents for ${\rm ECF}(x)$}

\begin{lem}\label{Lemma2}
In the ECF expansion, $q_k \geqslant q_{k-1} \geqslant 1$,
$p_{k+1} \geqslant p_k \geqslant 1$, and $q_k -p_k \geqslant
q_{k-1}-p_{k-1} \geqslant 1$ for every $k\geqslant 1$.
\end{lem}

\begin{proof}
Let $(x_n)$ be a sequence defined by $x_n =a_n x_{n-1}+ e_{n-1} x_{n-2}$
with $a_n$ an even positive integers and $e_n \in \{ \pm 1\}$. Suppose that
$x_{k_0} \geqslant x_{k_0-1} \geqslant 1$ for some $k_0 \geqslant 1$.
Then $x_{k_0+1} \geqslant 2x_{k_0} -x_{k_0-1} \geqslant x_{k_0}$.
This shows inductively that $x_n \geqslant x_{n-1} \geqslant 1$ for every
$n\geqslant k_0$. The statement follows by taking $(x_n,k_0)=(q_n,1)$,
$(x_n,k_0)=(p_n,2)$, and respectively $(x_n,k_0)=(q_n-p_n,1)$.
\end{proof}

Furthermore, since $p_{n-1} q_n -p_n q_{n-1}=\pm 1$, it follows that
$q_n (x) > q_{n-1} (x)$, for all $n\geqslant 2$ and $x\in \Omega$.

\begin{prop}\label{Prop3}
For each $x\in\Omega$ the following are equivalent:
\begin{enumerate}
\item[(i)]
$P/Q,P^\prime/Q^\prime$ are successive convergents in ${\rm ECF}(x)$.
\item[(ii)]
$M=\Big( \begin{smallmatrix} P & P^\prime \\ Q & Q^\prime \end{smallmatrix}\Big)
\in \RR_E$ and $0< \vert E_M (x)\vert <1$.
\end{enumerate}
\end{prop}

\begin{proof}
(i) $\Rightarrow $ (ii) Suppose
$M= \Big( \begin{smallmatrix} P & P^\prime \\ Q & Q^\prime \end{smallmatrix}\Big)
=\Big( \begin{smallmatrix} p_{n-1} & p_n \\ q_{n-1} & q_n \end{smallmatrix}\Big)$
for some $n\geqslant 1$. From Lemma \ref{Lemma2}, $\Big( \begin{smallmatrix}
0 & e_{k-1} \\ 1 & a_k \end{smallmatrix} \Big) \equiv J\pmod{2}$ and equality \eqref{2.3}
we infer that $M\in \RR_E$. The second condition in (ii) follows from \eqref{2.5}.

(ii) $\Rightarrow $ (i)
Consider first the case $Q=1$. Only the matrices $M=\Big( \begin{smallmatrix}
0 & 1 \\ 1 & Q^\prime \end{smallmatrix}\Big)$ and $M=\Big( \begin{smallmatrix}
1 & Q^\prime -1 \\ 1 & Q^\prime \end{smallmatrix}\Big)$ may arise. Since
$M\equiv I$ or $J\pmod{2}$, only the former case can occur and
$Q^\prime$ is necessarily an even positive integer. The corresponding inequality
\[
0< \bigg| \frac{Q^\prime x-1}{-x} \bigg| <1\quad \mbox{\rm is equivalent to}\quad
x\in\bigg( \frac{1}{Q^\prime +1},\frac{1}{Q^\prime}\bigg) \cup
\bigg( \frac{1}{Q^\prime},\frac{1}{Q^\prime -1}\bigg)
\] 
or, according to the definition of $a_1$, to $a_1=Q^\prime$,
showing that $0/1,1/Q^\prime$ are successive convergents of $x$.

When $Q>1$, take ($\ell \geqslant 1$):
\begin{equation*}
\begin{split}
& e_M =1,\ Q_0=Q^\prime -2\ell Q,\ P_0=P^\prime -2\ell P \quad \mbox{\rm \ if \ $[\lambda] =2\ell,$} \\
& e_M =-1,\ Q_0 =2\ell Q-Q^\prime,\ P_0=2\ell P-P^\prime
\ \mbox{\rm \  if \  $[\lambda ]=2\ell-1,$} \\
& M_0 =\left( \begin{matrix} P_0 & P \\ Q_0 & Q \end{matrix}\right).
\end{split}
\end{equation*}

In both cases one has $0<Q_0<Q$, $M=M_0
\left( \begin{smallmatrix} 0 & e_M \\ 1 & 2\ell \end{smallmatrix}\right)$, and so
$M_0 \equiv I$ or $J\pmod{2}$.
Since $Q^\prime >Q>Q_0$, the condition $0< \vert E_M \vert <1$ is equivalent to
$x$ lying between $\frac{P^\prime+P}{Q^\prime +Q}$ and
$\frac{P^\prime -P}{Q^\prime -Q}$, while $0< \vert E_{M_0} \vert <1$ is equivalent to $x$
lying between $\frac{P+P_0}{Q+Q_0}$ and $\frac{P-P_0}{Q-Q_0}$.
When $P/Q<P^\prime/Q^\prime$ the former implies the latter because of
\begin{equation*}
\begin{split}
\frac{P-P_0}{Q-Q_0} & =\frac{(2\ell+1)P-P^\prime}{(2\ell+1)Q-Q^\prime} < \frac{P}{Q} <
\frac{P^\prime+P}{Q^\prime+Q} < \frac{P^\prime}{Q^\prime} <\frac{P^\prime-P}{Q^\prime-Q} \\
 & \leqslant \frac{P+P_0}{Q+Q_0}=\frac{P^\prime-(2\ell-1)P}{Q^\prime-(2\ell-1)Q}
< \frac{P_0}{Q_0} =\frac{P^\prime-2\ell P}{Q^\prime -2\ell Q} \mbox{\rm
\quad when\ } [\lambda ]=2\ell ,
\end{split}
\end{equation*}
and of
\begin{equation*}
\begin{split}
\frac{P_0}{Q_0} & =\frac{2\ell P-P^\prime}{2\ell Q-Q^\prime}
< \frac{P+P_0}{Q+Q_0} =\frac{(2\ell+1)P-P^\prime}{(2\ell+1)Q-Q^\prime} < \frac{P}{Q}
< \frac{P^\prime+P}{Q^\prime+Q} < \frac{P^\prime}{Q^\prime} \\ & < \frac{P^\prime-P}{Q^\prime -Q}
\leqslant \frac{P-P_0}{Q-Q_0} =\frac{P^\prime -(2\ell-1)P}{Q^\prime -(2\ell-1)Q}
\mbox{\rm \quad when\ } [\lambda ]=2\ell-1 .
\end{split}
\end{equation*}
When $P^\prime/Q^\prime <P/Q$, analogous inequalities show that $0<\vert E_M\vert <1$ implies
$0 < \vert E_{M_0} \vert <1$.
Furthermore, the inequalities $0\leqslant P_0 \leqslant P$ follow from
$\vert P^\prime Q-PQ^\prime \vert =\vert PQ_0 -P_0 Q\vert =1$ and $P\geqslant 1$.
\end{proof}

\subsection{Successive convergents for ${\rm OCF}(x)$}

Denominators of successive convergents for ${\rm OCF}(x)$ satisfy (\cite[Eq. 2.10]{Rie})
\begin{equation}\label{3.2}
\begin{split}
r_n & :=  q_n/q_{n-1} \\ & = a_n +e_{n-1}[[(a_{n-1},e_{n-2}),(a_{n-2},e_{n-3}),
\ldots , (a_2,e_1),(a_1,\ast)]] \\ & \geqslant a_n -[[(3,-1),(3,-1),\ldots,(3,-1),(3,\ast)]] \\
& > a_n -[[(3,-1),(3,-1),(3,-1)\ldots]] \\ & = a_n -1+1/G =a_n -2+G.
\end{split}
\end{equation}
In the opposite direction one has
\begin{equation}\label{3.3}
r_n = a_n + \frac{e_{n-1}}{r_{n-1}} < a_n +\frac{e_{n-1}}{a_{n-1}-2+G} \leqslant
a_n +\frac{1}{G-1}=a_n +G .
\end{equation}
In particular \eqref{3.2} and \eqref{3.3} show that if $a_n \geqslant 3$, then $r_n > 1+G$, proving

\begin{lem}\label{Lemma4}
If $r_n \leqslant 2+g$ then $a_n=1$, and
in particular $e_n=1$ and 
\[
0<\frac{q_n x-p_n}{-q_{n-1}x+p_{n-1}} <1.
\]
\end{lem}

\begin{prop}\label{Prop5}
For each $x\in\Omega$ the following are equivalent:
\begin{enumerate}
\item[(i)]
$P/Q,P^\prime/Q^\prime$ are successive convergents in $\operatorname{OCF}(x)$.
\item[(ii)]
$M=\left( \begin{smallmatrix} P & P^\prime \\
Q & Q^\prime \end{smallmatrix}\right) \in\RR_O$ and one of the following two conditions holds:
\begin{enumerate}
\item[($\ast$)] $\ \lambda_M :=Q^\prime /Q > 2+g$ and $0<\vert E_M (x)\vert <1$.
\item[($\ast\ast$)] $\ g < \lambda_M \leqslant 2+g$ and
$0< E_M (x)<1$.
\end{enumerate}
\end{enumerate}
\end{prop}

\begin{proof}
(i) $\Rightarrow$ (ii) Suppose that there is $n\geqslant 1$ such that
\begin{equation}\label{3.4}
M=\left( \begin{matrix} 0 & e_0=1 \\ 1 & a_1 \end{matrix}\right)
\left( \begin{matrix} 0 & e_1 \\ 1 & a_2 \end{matrix} \right) \cdots
\left( \begin{matrix} 0 & e_{n-1} \\ 1 & a_n \end{matrix}\right) =
\left( \begin{matrix} p_{n-1} & p_n \\ q_{n-1} & q_n \end{matrix}\right) .
\end{equation}
Since $\left( \begin{smallmatrix} 0 & e_{i-1} \\ 1 & a_i \end{smallmatrix}\right)
\equiv \left( \begin{smallmatrix} 0 & 1 \\ 1 & 1 \end{smallmatrix} \right)
=A \pmod{2}$ and $\{ [I],[A],[B]\}$ forms a subgroup
of $\operatorname{SL}_2 (\Z / 2\Z)$, it follows that $M\equiv I,A,$ or $B\pmod{2}$.
The inequality $GQ^\prime >Q$ follows from \eqref{3.2}, while $0\leqslant P=p_{n-1} \leqslant Q=q_{n-1}$,
$0< P^\prime =p_n\leqslant Q^\prime =q_n$ are well-known (they follow as a result of the
$\operatorname{RCF}\longrightarrow \operatorname{OCF}$ algorithm or can be directly deduced from
$p_{n-1}q_n-p_n q_{n-1}=\pm 1$).
Properties ($\ast$) and ($\ast\ast$) follow from \eqref{2.4}, \eqref{2.5}, and from Lemma \ref{Lemma4}.

(ii) $\Rightarrow$ (i) Consider the partition $( g,\infty) =\SS_1 \cup \SS_2 \cup \SS_3$, where
\begin{equation*}
\begin{split}
\SS_1 & = (g,1) \cup (2+g,3) \cup (4+g,5) \cup \ldots ,\\
\SS_2 & =[1,2) \cup [3,4) \cup [5,6) \cup \ldots , \\
\SS_3 & =[ 2,2+g) \cup [4,4+g) \cup [6,6+g) \cup \ldots
\end{split}
\end{equation*}
For each matrix $M=\left( \begin{smallmatrix} P & P^\prime \\ Q & Q^\prime \end{smallmatrix} \right)
\in \RR_O$ with $\lambda=\lambda_M$, define
\begin{equation*}
k_M =\begin{cases}
2\ell-1 & \mbox{\rm if $\lambda\in\SS_2$, $[\lambda]=2\ell-1$, $\ell\geqslant 1$,} \\
2\ell+1 & \mbox{\rm if $\lambda\in\SS_1$,
$[\lambda ]=2\ell$, $\ell\geqslant 0$, and $\{ \lambda \} >g$,} \\
2\ell-1 & \mbox{\rm if $\lambda\in\SS_3$,
$[\lambda ]=2\ell$, $\ell\geqslant 1$, and $\{ \lambda \} < g.$}
\end{cases}
\end{equation*}
Note that
\begin{equation*}
k_M \geqslant 3 \ \Longleftrightarrow \ \lambda > 2+g =G^2.
\end{equation*}
We prove the following statement:

\begin{lem}\label{Lemma6}
Let $x\in\Omega$ and $M=\left( \begin{smallmatrix} P & P^\prime \\ Q & Q^\prime \end{smallmatrix} \right)
\in \RR_O$  with $\widetilde{Q}=\min\{ Q,Q^\prime\} >1$ and
satisfying {\em ($\ast$)} or {\em ($\ast\ast$)}. There exist $e_M \in \{ \pm 1\}$ and $M_0=\left( \begin{smallmatrix}
P_0 & P \\ Q_0 & Q \end{smallmatrix}\right)\in\RR_O$ such that
\begin{equation}\label{3.5}
M=M_0 \left( \begin{matrix} 0 & e_M \\ 1 & k_M \end{matrix}\right),
\end{equation}
$e_M +k_{M_0}\geqslant 2$, $M_0$ satisfies the corresponding property
{\em ($\ast$)} or {\em ($\ast\ast$)}, and
$\widetilde{Q}_0=\min\{ Q_0,Q\}\leqslant\widetilde{Q}$. Furthermore, if $\lambda=\lambda_M \in \SS_1 \cup \SS_2$,
then we can take $\widetilde{Q}_0 <
\widetilde{Q}$.
\end{lem}

{\sl Proof of Lemma \ref{Lemma6}.}
Consider the following integers:
\begin{equation*}
e_M =\begin{cases}
1 & \mbox{\rm if $\lambda \in\SS_2 \cup \SS_3,$} \\
-1 & \mbox{\rm if $\lambda \in \SS_1.$}
\end{cases}
\end{equation*}
\begin{equation*}
\begin{split}
Q_0 & =\begin{cases}
Q^\prime -(2\ell-1)Q & \mbox{\rm if $\lambda\in\SS_3$, $[\lambda] =2\ell$, $\ell\geqslant 1$, and
$\{ \lambda\} < g,$}\\
Q^\prime -(2\ell-1)Q & \mbox{\rm if $\lambda \in \SS_2$, $[\lambda]=2\ell-1$, $\ell\geqslant 1,$} \\
(2\ell+1)Q-Q^\prime & \mbox{\rm if $\lambda\in\SS_1$, $[\lambda] =2\ell$, $\ell\geqslant 0$,
and $\{ \lambda\} > g.$}
\end{cases} \\
& = \begin{cases}
( 1+\{ \lambda\})Q & \mbox{\rm if $\lambda\in\SS_3,$} \\
 \{ \lambda\} Q & \mbox{\rm if $\lambda \in \SS_2,$} \\
( 1-\{ \lambda \} ) Q
& \mbox{\rm if $\lambda \in \SS_1.$}
\end{cases}
\end{split}
\end{equation*}
\begin{equation*}
P_0 =\begin{cases}
P^\prime -(2\ell-1)P & \mbox{\rm if $\lambda\in\SS_3$, $[\lambda] =2\ell$, $\ell\geqslant 1$, and
$\{ \lambda\} < g,$} \\
P^\prime -(2\ell-1)P & \mbox{\rm if $\lambda\in\SS_2$, $[\lambda]=2\ell-1$, $\ell\geqslant 1,$} \\
(2\ell+1)P-P^\prime & \mbox{\rm if $\lambda\in\SS_1$, $[\lambda] =2\ell$, $\ell\geqslant 0$,
and $\{ \lambda\} > g.$}
\end{cases}
\end{equation*}

Equality \eqref{3.5} holds in all cases with this choice for $Q_0$ and $P_0$.
One plainly checks that
\begin{equation*}
\lambda_0:=\frac{Q}{Q_0} \in \begin{cases}
(2+g,\infty) & \mbox{\rm if $\lambda \in\SS_1,$} \\
(1,\infty) & \mbox{\rm if $\lambda \in \SS_2 ,$} \\
(g,1] & \mbox{\rm if $\lambda\in\SS_3.$}
\end{cases}
\end{equation*}
In particular this shows that $\lambda_0 > g$. The inequality
$e_M +k_{M_0}\geqslant 2$ is trivial when $\lambda\in\SS_2\cup\SS_3$. When $\lambda\in\SS_1$
we have $\lambda_0 > 2+g$, hence $k_{M_0}\geqslant 3$ and
$e_M +k_{M_0}\geqslant 2$.

Clearly
$\left( \begin{smallmatrix} 0 & e_M \\ 1 & k_M \end{smallmatrix}\right)
\equiv A\pmod{2}$. The inequalities $0\leqslant P_0 \leqslant Q_0$ follow immediately
from $P_0 Q-PQ_0 =\pm 1$ and $P<Q$, the latter one being a consequence of
the assumption $\widetilde{Q} >1$.
The fact that $M_0$ satisfies either ($\ast$) or ($\ast\ast$) follows from Lemma \ref{Lemma7}.
$\square$

Back to the proof of Proposition \ref{Prop5}, note that when
$\lambda \in ( g,1]$ one has $0< Q_0 =Q-Q^\prime < Q^\prime <Q$
(the first inequality holds because $G<2$), while for $\lambda \in (\SS_1 \cup \SS_2)
\setminus ( g,1)$ it is plain that $0< Q_0 <Q<Q^\prime$. Hence
whenever $\lambda\in \SS_1 \cup \SS_2$ one has $\min\{ Q_0,Q\} < \min\{ Q,Q^\prime\}$.

When $\lambda\in\SS_3$ one only has $\min\{ Q_0,Q\}=\min\{ Q,Q^\prime\}$
(actually $Q<Q_0 <Q^\prime$). However, in this case $e_M =-1$ so $k_{M_0}\geqslant 3$,
and $\lambda_{M_0}=Q/Q_0 \in ( g,1)$. Thus one can apply
the same procedure to $M_0$ and find $M_{-1}=\left( \begin{smallmatrix} P_{-1} & P_0 \\
Q_{-1} & Q_0 \end{smallmatrix} \right)\in \RR_0$ that satisfies ($\ast$) or ($\ast\ast$),
and such that $M_0=M_{-1} \left(
\begin{smallmatrix} 0 & e_{M_0} \\ 1 & k_{M_0} \end{smallmatrix}\right)$,
$e_{M_0}+k_{M_{-1}} \geqslant 2$, and $\widetilde{Q}_{-1} :=\min
\{ Q_{-1},Q_0\} < \widetilde{Q}_0 =\widetilde{Q}$ (this inequality is strict because
$\lambda_0 \in (g,1) \subseteq \SS_1$).

We next discuss the case $\widetilde{Q}=1$. When $Q^\prime =1 \leqslant Q$, the inequality
$Q^\prime /Q=1/Q \geqslant g$ yields $Q=1$. Hence
$M=\left( \begin{smallmatrix} 0 & 1 \\ 1 & 1 \end{smallmatrix} \right)$,
with $0/1,1/1$ successive convergents of every $x\in (0,1)$
that satisfies $0< \frac{x-1}{-x} <1$, i.e. of every $x\in ( 1/2,1)$.
Suppose now $Q=1<Q^\prime$. When $1/G<Q^\prime /Q=Q^\prime < 2+g$,
one has $Q^\prime =2$ and only the matrices
$M=\left( \begin{smallmatrix} 0 & 1 \\ 1 & 2 \end{smallmatrix}\right)$
and $M=\left( \begin{smallmatrix} 1 & 1 \\ 1 & 2 \end{smallmatrix}\right)$ may
arise. But the former matrix is not admissible being $\equiv \left( \begin{smallmatrix}
0 & 1 \\ 1 & 0 \end{smallmatrix} \right)\pmod{2}$, while the latter matrix corresponds to
$0< \frac{2x-1}{-x+1} <1$, hence $x\in ( 1/2,2/3)$,
$e_1 =1$ and $a_1 =[ 1/x] =1$, and indeed $1/1,1/2$
are successive convergents in ${\rm OCF}(x)$ for every $x\in ( 1/2,2/3)$.
When $2+g < Q^\prime/Q=Q^\prime$ the only matrices that may arise
are $M=\left( \begin{smallmatrix} 0 & 1 \\ 1 & Q^\prime \end{smallmatrix}\right)$
with $Q^\prime \geqslant 3$ odd, and respectively $M=\left( \begin{smallmatrix} 1 & Q^\prime -1 \\
1 & Q^\prime \end{smallmatrix}\right)$ with $Q^\prime \geqslant 4$ even.
The inequality for the former is 
\[
0<\bigg| \frac{Q^\prime x-1}{-x}\bigg| <1,\quad
\mbox{\rm which gives}\quad x\in \bigg( \frac{1}{Q^\prime +1},\frac{1}{Q^\prime}\bigg) \cup
\bigg( \frac{1}{Q^\prime},\frac{1}{Q^\prime -1}\bigg)
\] 
with $Q^\prime$ odd, so that
$a_1 =Q^\prime$ (and $e_1 =1$ respectively $e_1=-1$). The inequality for the
latter is 
\[
0< \bigg| \frac{Q^\prime x-Q^\prime+1}{-x+1}\bigg| < 1,\quad \mbox{\rm  giving}\quad
\frac{Q^\prime}{Q^\prime +1} > x > \frac{Q^\prime -2}{Q^\prime -1} \geqslant \frac{2}{3},
\]
so $e_1=1$, $a_1 =1$. Furthermore one has 
\[
1/Q^\prime < 1/x-1 =T_D (x) <1/(Q^\prime -2)
\] 
with $Q^\prime -1\geqslant 3$ odd integer, so $a_2 =Q^\prime -1$
and $M=\left( \begin{smallmatrix} 0 & 1 \\ 1 & 1 \end{smallmatrix}\right)
\left( \begin{smallmatrix} 0 & 1 \\ 1 & Q^\prime -1 \end{smallmatrix} \right)$, showing that indeed
$1/1,(Q^\prime -1)/Q^\prime$ are successive convergents in ${\rm OCF}(x)$ for every $x$
with $\frac{Q^\prime}{Q^\prime +1} > x > \frac{Q^\prime -2}{Q^\prime -1}$ and $Q^\prime \geqslant 4$ even.

This inductive process on $\widetilde{Q}$ now implies that
 \eqref{3.4} holds for some $e_1,\ldots,$ $e_{n-1}\in \{\pm 1\}$ and
$a_1,\ldots,a_n$ odd positive integers with $e_i+a_i \geqslant 2$, $\forall i
\in \{ 1,\ldots,n-1\}$.
Conditions ($\ast$) and ($\ast\ast )$ show that $x$ lies between $\frac{p_{n}-p_{n-1}}{q_{n}-q_{n-1}}$
and $\frac{p_{n}+p_{n-1}}{q_{n}+q_{n-1}}$
when $q_{n} > q_{n-1}$, and between $\frac{p_{n}}{q_{n}}$ and
$\frac{p_{n}+p_{n-1}}{q_{n}+q_{n-1}}$ when $q_{n}<q_{n-1}$. So $x$ is of the form
$[[(a_1,e_1),(a_2,e_2),\ldots, (a_{n-1},e_{n-1}),(a_n+t,\ast)]]$ for some
$t\in (-1,1)$ when $q_{n}>q_{n-1}$, and respectively $t\in (0,1)$ when
$q_{n}<q_{n-1}$. Therefore $p_{n-1}/q_{n-1}=P/Q$,
$p_n/q_n=P^\prime /Q^\prime$ are successive convergents of $x$.
\end{proof}

\begin{lem}\label{Lemma7}
With the definitions from the proof of implication {\em (ii) $\Rightarrow$ (i)} in
Proposition {\em \ref{Prop5}}, one has:
\begin{enumerate}
\item[(i)]
If $g<\lambda <1$, then $0<E_M(x) <1
\Rightarrow \vert E_{M_0} (x)\vert  <1$.
\item[(ii)]
If $1\leqslant \lambda <2+g$, then $0< E_M (x) <1
\Rightarrow 0< E_{M_0} (x) <1$.
\item[(iii)]
If $2\ell+g <\lambda < 2\ell+1$, $\ell \geqslant 1$, then
$\vert E_M (x)\vert  <1 \Rightarrow -1 < E_{M_0} (x) <0$.
\item[(iv)]
If $2\ell -1 \leqslant \lambda < 2\ell+g$, $\ell \geqslant 2$, then
$\vert E_M (x)\vert <1 \Rightarrow 0< E_{M_0} (x) <1$.
\end{enumerate}
\end{lem}

\begin{proof}
In all cases $0<E_M(x)=\frac{Q^\prime x-P^\prime}{-Qx+P} <1$ is equivalent with $x$ lying between
$\frac{P^\prime}{Q^\prime}$ and $\frac{P^\prime+P}{Q^\prime+Q}$, while
$0<E_{M_0} (x)=\frac{Qx-P}{-Q_0 x+P_0} <1$ is equivalent to $x$ lying between $\frac{P}{Q}$ and $\frac{P+P_0}{Q+Q_0}$.

(i) In this case
$Q_0=Q-Q^\prime < Q$ and so $-1< E_{M_0}(x) <1$ is equivalent to $x$ lying
between $\frac{P+P_0}{Q+Q_0}=\frac{2P-P^\prime}{2Q-Q^\prime}$ and
$\frac{P-P_0}{Q-Q_0}=\frac{P^\prime}{Q^\prime}$. The conclusion follows because
\begin{equation*}
\begin{split}
& \frac{2P-P^\prime}{2Q-Q^\prime} < \frac{P}{Q} < \frac{P^\prime+P}{Q^\prime+Q}
< \frac{P^\prime}{Q^\prime}\quad \mbox{\rm when} \ \frac{P}{Q}<\frac{P^\prime}{Q^\prime}, \quad \mbox{\rm and} \\
& \frac{P^\prime}{Q^\prime} < \frac{P^\prime+P}{Q^\prime+Q} < \frac{P}{Q}
< \frac{2P-P^\prime}{2Q-Q^\prime} \quad \mbox{\rm when} \  \frac{P^\prime}{Q^\prime} < \frac{P}{Q}.
\end{split}
\end{equation*}

(ii)  In this case $\frac{P+P_0}{Q+Q_0}=\frac{P^\prime}{Q^\prime}$ and $x$ between
$\frac{P^\prime}{Q^\prime}$ and $\frac{P^\prime+P}{Q^\prime +Q}$ implies $x$ between
$\frac{P}{Q}$ and $\frac{P^\prime}{Q^\prime}$.

(iii) In this case $0<Q_0=(2\ell+1)Q-Q^\prime <Q< Q^\prime$, and $-1<E_M (x)<1$ is equivalent to
$x$ lying between $\frac{P^\prime+P}{Q^\prime+Q}$ and $\frac{P^\prime-P}{Q^\prime-Q}$, while
$-1<E_{M_0} (x) <0$ is equivalent to $x$ lying between $\frac{P}{Q}$ and
$\frac{P-P_0}{Q-Q_0}=\frac{P^\prime -2\ell P}{Q^\prime-2\ell Q}$. The implication follows because either 
\[
\frac{P}{Q} < \frac{P^\prime+P}{Q^\prime+Q} < \frac{P^\prime}{Q^\prime} <
\frac{P^\prime -P}{Q^\prime -Q} < \frac{P^\prime -2\ell P}{Q^\prime -2\ell Q}
\]
or
\[
\frac{P^\prime -2\ell P}{Q^\prime -2\ell Q} < \frac{P^\prime -P}{Q^\prime -Q} < \frac{P^\prime}{Q^\prime}
< \frac{P^\prime+P}{Q^\prime +Q} < \frac{P}{Q}.
\]

(iv) In this case $Q^\prime >Q$ and $\frac{P+P_0}{Q+Q_0}=\frac{P^\prime -(2\ell-2)P}{Q^\prime -(2\ell-2)Q}$.
The implication follows because $-1<E_M (x)<1$ is equivalent with $x$ lying between
$\frac{P^\prime+P}{Q^\prime+Q}$ and $\frac{P^\prime-P}{Q^\prime-Q}$,
$0<E_{M_0} (x)<1$ is equivalent with $x$ lying between $\frac{P}{Q}$ and
$\frac{P+P_0}{Q+Q_0}$, and either 
\[
\frac{P}{Q}< \frac{P^\prime+P}{Q^\prime+Q}
<\frac{P^\prime}{Q^\prime} < \frac{P^\prime-P}{Q^\prime-Q} <
\frac{P^\prime -(2\ell-2)P}{Q^\prime -(2\ell-2)Q}
\]
or
\[ 
\frac{P^\prime -(2\ell-2)P}{Q^\prime -(2\ell -2)Q} < \frac{P^\prime -P}{Q^\prime -Q}
< \frac{P^\prime}{Q^\prime} < \frac{P^\prime+P}{Q^\prime+Q} <\frac{P}{Q}.
\]
\end{proof}

The following statement will also be useful:

\begin{lem}\label{Lemma8}
Denominators of successive convergents in OCF satisfy
\begin{enumerate}
\item[(i)]
$q_{n+2}>q_n$.
\item[(ii)]
$q_{n+3} >q_n$.
\item[(iii)]
$q_{n+2} > \min\{ q_n,q_{n+1}\}$.
\end{enumerate}
\end{lem}

\begin{proof}
By Proposition \ref{Prop5} and its proof
$q_{n+2}/q_{n+1} >2 \Rightarrow q_{n+2}/q_n > 2g >1$,
$q_{n+2}/q_{n+1} \in (1,2) \Rightarrow q_{n+1}/q_n >1
\Rightarrow q_{n+2}/q_n >1$, and $q_{n+2}/q_{n+1} \in
( g,1) \Rightarrow q_{n+1}/q_n > 2+g
\Rightarrow q_{n+2}/q_n > g ( 2+g) >1$.
Thus in all possible cases $q_{n+2}>q_n$, which establishes (i).

(ii) follows from
$q_{n+3} / q_{n+2} \in ( g,1) \Rightarrow
q_{n+2}/q_{n+1} > 2+g \Rightarrow
q_{n+3}/q_n > ( 2+g) g^2 =1$,
$q_{n+3}/q_{n+2} =\lambda \in (1,2) \Rightarrow
q_{n+2}/q_{n+1} =1/(\lambda -1) \Rightarrow
q_{n+3}/q_n > \lambda g/(\lambda -1) >2g >1$, $q_{n+3}/q_{n+2} \in ( 2,2+g)
\Rightarrow q_{n+2}/q_{n+1} \in ( g,1) \Rightarrow
q_{n+1}/q_n > 2+g \Rightarrow q_{n+3}/q_n > 2g ( 2+g) >1$, and
$q_{n+3}/q_n > 2+g \Rightarrow q_{n+3}/ q_n > ( 2+g) g^2=1$.

To prove (iii) suppose that $q_{n+2} \leqslant q_{n+1}$. Then
$q_{n+2} / q_{n+1} \in ( g,1)$, which gives in turn
$q_{n+1} / q_n > 2+g$, and therefore $q_{n+2} / q_n > g ( 2+g) >1$.
\end{proof}

\begin{xrem}
Proposition \ref{Prop3} was originally proved, using a different method, by Kraaikamp
and Lopes \cite{KL}, but Proposition \ref{Prop5} is, to the best of our research, new.  Our proofs have
an additional benefit of implying how to derive $a_n$ and $e_{n-1}$ (and hence
$q_{n-2}$) if only $q_{n-1}$ and $q_n$ are known.

Our investigations yielded yet another method of proof, significantly longer but
more direct, which we sketch here.  Examples 1.8 in \cite{Mas} explain how to
algorithmically generate the OCF expansion of $x$ from the RCF expansion of $x$
using \emph{insertion},
\begin{align*}\text{replacing
}&[[\ldots,(a_n,1),(a_{n+1},e_{n+1}),\ldots]]\\ \text{ with }
&[[\ldots,(a_n+1,-1),(1,1),(a_{n+1}-1,e_{n+2}),\ldots]],\end{align*} 
and \emph{singularization}, 
\begin{align*}\text{replacing
}&[[\ldots,(a_n,e_n),(1,1),(a_{n+2},e_{n+2}),\ldots]]\\ \text{ with }
&[[\ldots,(a_n+e_n,-e_n),(a_{n+2}+1,e_{n+2}),\ldots]].
\end{align*}
Both of these operations alter the sequence of convergents: insertion adds a new convergent, while
singularization deletes one.  Nevertheless, it can be shown that if
$P/Q,P'/Q'$ are successive RCF convergents to some $x$, then either
$P/Q,P'/Q'$ are successive OCF convergents to $x$, or
$(Q-P)/Q,(Q'-P')/Q'$ are successive OCF convergents
to $1-x$. (Only one of these pairs forms a matrix that is congruent to $I$, $A$, or
$B$ modulo $2$.) By carefully following how insertion and singularization change the
last $e_{n-1}$ and $a_n$ in the RCF expansion of $P'/Q'$ into the last
$e_{m-1}$ and $a_m$ of the OCF expansion of $P'/Q'$, we can determine
exactly what $e(M)$ and $a(M)$ must be and hence how to derive $P_0$ and $Q_0$.
A similar proof works for the ECF case as well.
\end{xrem}

\section{Estimating the limiting joint distribution for ECF and OCF}
For each $M=\left( \begin{smallmatrix} P & P^\prime \\ Q & Q^\prime \end{smallmatrix}\right)\in\RR$ and
$\xi\in (0,1]$ denote by $I_\xi^+(M)$ (respectively, $I_\xi^-(M)$) the set of
solutions $x$ of $0\leqslant E_M(x)\leqslant\xi$ (respectively, of $-\xi \leqslant E_M(x)\leqslant 0$). The Lebesgue measure of
$I_\xi^\pm (M)$ is
\begin{equation*}
f_\xi^\pm (Q,Q^\prime) =\left| \frac{P^\prime \pm \xi P}{Q^\prime \pm \xi Q} -\frac{P^\prime}{Q^\prime}\right|
=\frac{\xi}{Q^\prime (Q^\prime \pm \xi Q)} .
\end{equation*}

The integral
\begin{equation*}
\begin{split}
F_\pm  =F_\pm (x_1,x_2,x_3,x_4) := & \int_{R/x_2}^\infty dv
\int_0^{\min\{ x_3 v,x_1 R\}} du\, f^\pm_{x_4}(u,v) \\
= & \pm\int_{R/x_2}^\infty \frac{dv}{v}\, \log
\bigg| \frac{v\pm x_4 \min\{ x_3 v,x_1 R\}}{v} \bigg| \\ = &
\pm \int_{x_3/x_2}^\infty \frac{dw}{w} \, \log \bigg| \frac{w\pm x_3 x_4 \min\{ w,x_1\}}{w}\bigg|
\end{split}
\end{equation*}
can be expressed when $x_3\geqslant x_1x_2$ as
\begin{equation*}
F_\pm  = \pm \int_0^{x_1x_2x_4} \frac{dt}{t}\, \log (1\pm t)  = \mp\Li_2 (\mp x_1x_2x_4) ,
\end{equation*}
and when $x_3 < x_1 x_2$ as
\begin{equation*}
\begin{split}
F_\pm & = \int_{x_3/x_2}^{x_1} \frac{dw}{w}\, \log (1\pm x_3x_4) \pm
\int_{x_1}^\infty \frac{dw}{w}\, \log \frac{w\pm x_1x_3x_4}{w} \\
& = \pm\log (1\pm x_3x_4)\log\frac{x_1x_2}{x_3} \mp \Li_2 (\mp x_3 x_4),
\end{split}
\end{equation*}
so $F_\pm$ is as in \eqref{1.6}.

\subsection{The ECF case}
By Lemma \ref{Lemma2} and Proposition \ref{Prop3}, for each $R>1$ and $x\in\Omega$ there is a unique $M=
\left( \begin{smallmatrix} P & P^\prime \\ Q & Q^\prime \end{smallmatrix}\right)\in\RR_E$
with $Q\leqslant R<Q^\prime$ and $\vert E_M (x) \vert <1$.
Given $x_1,x_2,x_3,x_4\in (0,1)$ consider $\NN^{E,\pm}_{x_1,x_2,x_3,x_4} (x,R)$, the number
of matrices $M\in \RR_E$ that satisfy \eqref{1.1} and \eqref{1.2}. One has
\begin{equation*}
\LL^{E,\pm}(R) =\LL^{E,\pm}_{x_1,x_2,x_3,x_4} (R)=\int_0^1 \NN^{E,\pm}_{x_1,x_2,x_3,x_4} (x,R)\, dx .
\end{equation*}

For $\Gamma \in \{ I,J,A,B\}$ we shall estimate
\begin{equation*}
\LL_\Gamma^\pm (R):= \sum\limits_{\substack{M=\left( \begin{smallmatrix}
P & P^\prime \\ Q & Q^\prime \end{smallmatrix}\right) \in \RR_E \\
Q^\prime \geqslant R/x_2 \\ Q \leqslant \min\{ x_3 Q^\prime,x_1 R\} \\
M\equiv \Gamma \hspace{-6pt}\pmod{2}}} f^\pm_{x_4}(Q,Q^\prime).
\end{equation*}
This can be done by M\" obius summation, as in the following standard lemmas (for Lemma \ref{Mob2} see, e.g.,
\cite[Lemma 2.1]{BG}).

\begin{lem}\label{Mob1}
For every interval $J$, every function $g\in C^1 (J)$ of total variation $T_J g$, and every integer $x$,
with $\sigma_0$ the divisor counting function,
\begin{equation*}
\sum\limits_{\substack{a\in J,\, b\in [1,q] \\ ab \equiv x \hspace{-6pt} \pmod{q} \\
(a,q)=1}} g(a) = \sum\limits_{\substack{a\in J \\ (a,q)=1}} g(a)=\frac{\varphi(q)}{q} \int_J g(u)\, du +O\big( \sigma_0 (q)
(\| g\|_\infty +T_J g)\big) .
\end{equation*}
\end{lem}

\begin{lem}\label{Mob2}
For every interval $J$, every $V\in C^1 [0,N]$, and every $\ell \in \N$,
\begin{equation*}
\sum\limits_{\substack{1\leqslant q\leqslant N \\ (q,\ell)=1}} \frac{\varphi(q)}{q}\, V(q) =
C(\ell) \int_0^N V(u)\, du +O_\ell \big( (\|V\|_\infty +T_0^N V)\log N\big),
\end{equation*}
with
\begin{equation*}
C(\ell) =\frac{1}{\zeta(2)} \prod\limits_{\substack{p\in \PP \\ p\mid \ell}}
\bigg( 1+\frac{1}{p}\bigg)^{-1} .
\end{equation*}
\end{lem}

Changing $b$ to $q-b$ in Lemma \ref{Mob1}, we infer
\begin{cor}\label{Mob3}
Suppose $q$ is an odd positive integer.
For every interval $J$, every $g\in C^1 (J)$, and every integer $x$,
\begin{equation*}
\sum\limits_{\substack{a\in J,\, b\in [1,q/2] \\ ab \equiv x\, or\, -x
\hspace{-6pt} \pmod{q} \\ (a,q)=1}} g(a) = \frac{\varphi(q)}{q} \int_J g(u)\, du +
O\big( (\| g\|_\infty +T_J g)\sigma_0(q)\big).
\end{equation*}
\end{cor}

Since $P^\prime Q-PQ^\prime =\pm 1$, $P^\prime,Q$ even and $Q^\prime$ odd entail
$P$ odd, we infer (with $Q=2q$, $P^\prime =2p^\prime$, $\bar{x}$ the multiplicative inverse
of $x \pmod{Q^\prime}$)
\begin{equation}\label{4.1}
\begin{split}
\LL_I^\pm (R) & = \sum\limits_{\substack{Q^\prime \geqslant R/x_2 \\ Q^\prime \equiv 1
\hspace{-6pt}\pmod{2}}} \sum\limits_{\substack{q\in [1,\min\{ x_3 Q^\prime,y_1 R\}/2] \\
p^\prime \in [1,Q^\prime /2] \\ p^\prime q \equiv \pm \overline{4} \hspace{-6pt} \pmod{Q^\prime}}}
f_{x_4}^\pm (2q,Q^\prime) \\
& = \sum\limits_{\substack{Q^\prime \geqslant R/x_2 \\ Q^\prime \equiv 1
\hspace{-6pt}\pmod{2}}} \Bigg( \frac{\varphi(Q^\prime)}{Q^\prime}
\int_0^{\min\{ x_3 Q^\prime ,x_1 R\}/2} f_{x_4}^\pm (2q,Q^\prime)\, dq
+O_\varepsilon \big( Q^{\prime -2+\varepsilon}\big)\Bigg) \\
& = \frac{1}{2} \sum\limits_{\substack{Q^\prime \geqslant R/x_2 \\ Q^\prime \equiv 1
\hspace{-6pt}\pmod{2}}} \frac{\varphi(Q^\prime)}{Q^\prime}
\int_0^{\min\{ x_3 Q^\prime, x_1 R\}} f_{x_4}^\pm (u,Q^\prime)\, du
+O_\varepsilon \big( R^{-1+\varepsilon}\big) \\
& =\frac{C(2) F_\pm}{2} +O_\varepsilon \big( R^{-1+\varepsilon}\big) =
\frac{F_\pm}{3\zeta(2)} +O_\varepsilon \big( R^{-1+\varepsilon}\big) .
\end{split}
\end{equation}

On the other hand, we have that $P^\prime Q-PQ^\prime =\pm 1$ and $Q^\prime$ even entail that
both $Q$ and $P^\prime$ are odd, and the condition $P$ even is equivalent to
$P^\prime Q \equiv \pm 1 \pmod{2Q^\prime}$. Since in this case $\varphi (2Q^\prime)=2\varphi (Q^\prime)$,
we infer
\begin{equation*}
\begin{split}
\LL_J^\pm (R) & = \sum\limits_{\substack{Q^\prime \geqslant R/x_2 \\ Q^\prime \equiv 0
\hspace{-6pt}\pmod{2}}} \sum\limits_{\substack{Q\in [1,\min\{ x_3 Q^\prime,x_1 R\}] \\
P^\prime\in [1,Q^\prime] \\ P^\prime Q \equiv \pm 1 \hspace{-6pt}\pmod{2Q^\prime}}} f_{x_4}^\pm (Q,Q^\prime) \\
& = \sum\limits_{\substack{Q^\prime \geqslant R/x_2 \\ Q^\prime \equiv 0
\hspace{-6pt}\pmod{2}}} \Bigg( \frac{\varphi(2Q^\prime)}{2Q^\prime}
\int_0^{\min\{ x_3 Q^\prime,x_1 R\}} f_{x_4}^\pm (u,Q^\prime)\, du +O_\varepsilon
\big( Q^{\prime -2+\varepsilon}\big) \Bigg) \\ &
= \bigg( \frac{1}{\zeta(2)}-C(2)\bigg) F_\pm +O_\varepsilon \big( R^{-1+\varepsilon}\big)
=\frac{F_\pm}{3\zeta(2)} +O_\varepsilon \big( R^{-1+\varepsilon}\big),
\end{split}
\end{equation*}
leading to
\begin{equation*}
\LL^{E,\pm}(R) =\LL_I^\pm (R)+\LL_J^\pm (R) = \frac{2F_\pm}{3\zeta(2)} +O_\varepsilon \big( R^{-1+\varepsilon}\big) ,
\end{equation*}
and concluding the proof of \eqref{1.3}.

The corresponding estimates for $\LL_B^\pm (R)$ and $\LL_A^\pm (R)$ are useful for the OCF situation.
To estimate $\LL_B^\pm (R)$, note that $P^\prime Q-PQ^\prime =\pm 1$ and $Q^\prime$ even
entail that both $P^\prime$ and $Q$ are odd, $\varphi(2Q^\prime)=2\varphi(Q^\prime)$, and thus
\begin{equation}\label{4.2}
\begin{split}
\LL_B^\pm (R) & = \sum\limits_{\substack{Q^\prime \geqslant R/x_2 \\ Q^\prime \equiv 0\hspace{-6pt}\pmod{2}}}
\sum\limits_{\substack{Q\in [1,\min\{ x_3 Q^\prime ,x_1 R\}] \\ P^\prime \in [1,Q^\prime],\,
P^\prime Q \equiv \pm 1 \hspace{-6pt}\pmod{Q^\prime} \\ \frac{P^\prime Q\mp 1}{Q^\prime}
\equiv 1\hspace{-6pt} \pmod{2}}} f_{x_4}^\pm (Q,Q^\prime)  \\
& = \sum\limits_{\substack{Q^\prime \geqslant R/x_2 \\ Q^\prime \equiv 0\hspace{-6pt}\pmod{2}}}
\hspace{-5pt} \Bigg( \sum\limits_{\substack{Q\in [1,\min\{ x_3 Q^\prime ,x_1 R\}] \\ P^\prime \in [1,Q^\prime],\,
P^\prime Q \equiv \pm 1 \hspace{-6pt}\pmod{Q^\prime}}} \hspace{-20pt} f_{x_4}^\pm (Q,Q^\prime)  - \hspace{-20pt}
\sum\limits_{\substack{Q\in [1,\min\{ x_3 Q^\prime ,x_1 R\}] \\ P^\prime \in [1,Q^\prime],\
P^\prime Q \equiv \pm 1 \hspace{-6pt}\pmod{2Q^\prime}}} \hspace{-20pt} f_{x_4}^\pm (Q,Q^\prime)  \Bigg) \\
& =\sum\limits_{\substack{Q^\prime \geqslant R/x_2 \\ Q^\prime \equiv 0\hspace{-6pt}\pmod{2}}} \hspace{-5pt}
\Bigg( \bigg( \frac{2\varphi(Q^\prime)}{Q^\prime} -\frac{\varphi(2Q^\prime)}{2Q^\prime} \bigg)
\int_0^{\min\{ x_3 Q^\prime ,x_1 R\}} \hspace{-20pt} f_{x_4}^\pm (u,Q^\prime) \, du +O_\varepsilon \big( Q^{\prime -2+\varepsilon}\big)
\Bigg) \\ & = \sum\limits_{\substack{Q^\prime \geqslant R/x_2 \\ Q^\prime \equiv 0\hspace{-6pt}\pmod{2}}} \hspace{-5pt}
\Bigg( \frac{\varphi(Q^\prime)}{Q^\prime} \int_0^{\min\{ x_3 Q^\prime , x_1 R\}} \hspace{-20pt} f_{x_4}^\pm (u,Q^\prime)\, du
+O_\varepsilon \big( Q^{\prime -2+\varepsilon}\big) \Bigg) \\ &
= \bigg( \frac{1}{\zeta(2)} - C(2)\bigg) F_\pm +O_\varepsilon \big( R^{-1+\varepsilon}\big)
=\frac{F_\pm}{3\zeta(2)} +O_\varepsilon \big( R^{-1+\varepsilon}\big) .
\end{split}
\end{equation}

Finally, $P^\prime Q-PQ^\prime =\pm 1$ and $P$ even entail that both $P^\prime$ and $Q$ are odd, and so
\begin{equation}\label{4.3}
\begin{split}
\LL_A^\pm (R) & = \sum\limits_{\substack{Q^\prime \geqslant R/x_2 \\ Q^\prime \equiv 1\hspace{-6pt}\pmod{2}}}
\sum\limits_{\substack{Q\in [1,\min\{ x_3 Q^\prime,x_1 R\}] \\ P^\prime\in [1,Q^\prime],\,
P^\prime Q \equiv \pm 1 \hspace{-6pt} \pmod{2Q^\prime}}} f_{x_4}^\pm (Q,Q^\prime) \\
& = \sum\limits_{\substack{Q^\prime \geqslant R/x_2 \\ Q^\prime \equiv 1\hspace{-6pt}\pmod{2}}}
\Bigg( \frac{\varphi (2Q^\prime)}{2Q^\prime} \int_0^{\min\{ x_3 Q^\prime,x_1 R\}} f_{x_4}^\pm (u,Q^\prime)\, du
+O_\varepsilon \big( Q^{\prime -2+\varepsilon}\big) \Bigg) \\
& = \frac{C(2)}{2}\, F_\pm +O_\varepsilon \big( R^{-1+\varepsilon}\big) =
\frac{F_\pm }{3\zeta(2)} +O_\varepsilon \big( R^{-1+\varepsilon}\big) .
\end{split}
\end{equation}

\subsection{The OCF case.} This requires more caution as the sequence
of denominators of successive convergents is not monotonically increasing in general.
We wish to characterize those matrices $M\in\RR_O$ for which
$P/Q,P^\prime /Q^\prime$ are successive convergents of $x\in\Omega$ and
$Q=q_{n_R} \leqslant R< Q^\prime =q_{n_R+1}$.
A priori, Lemma \ref{Lemma8} shows that for each $R>1$ there is at least one pair
and at most two pairs $(Q,Q^\prime)$ of denominators of successive convergents
of $x$ with $Q\leqslant R<Q^\prime$. Moreover, if there are two such pairs $(Q,Q^\prime)$, then
they must be of the form $(q_{n_R},q_{n_R+1})$ or $(q_{n_R+2},q_{n_R+3})$.
We wish to precisely distinguish $n_R$ from $n_R+2$. Because all predecessors of $Q_0$ in the sequence of denominators
of $OCF$ convergents are $<Q$ by Lemma \ref{Lemma8}, equality $(Q,Q^\prime)=(q_{n_R},q_{n_R+1})$ occurs exactly when
\begin{equation*}
Q \leqslant R < Q^\prime \quad \mbox{\rm and} \quad R>Q_0 .
\end{equation*}
Note that if $\lambda =Q^\prime /Q\in\SS_1 \cup \SS_2$, then necessarily
$Q>Q_0$. Furthermore, if $\lambda \in\SS_3$, then $Q<Q_0$. The
contribution of those pairs $(Q,Q^\prime)$ with $\lambda\in\SS_3$ and $Q_0=Q(1+\{ \lambda\})> R$
should be subtracted, and so we can write
\begin{equation*}
\begin{split}
\LL^{O,+} (R) & =\LL_I^+(R)+\LL_A^+(R)+\LL_B^+(R)-\DD_1 (R), \\
\LL^{O,-} (R) & =\LL_I^-(R)+\LL_A^-(R)+\LL_B^-(R)-\DD_2 (R)-\DD_3 (R),
\end{split}
\end{equation*}
with
\begin{equation*}
\begin{split}
\DD_1 (R) & = \sum\limits_{\substack{M\in\RR_O,\, Q^\prime > R/x_2 \\
Q\leqslant \min\{ x_3 Q^\prime ,x_1 R\} \\ \lambda=Q^\prime/Q \in \SS_3,\,
Q(1+\{ \lambda\}) > R}} \hspace{-10pt} f_{x_4}^+ (Q,Q^\prime) =
\sum_{\ell \geqslant 1} \sum\limits_{\substack{M\in\RR_O,\, Q^\prime > R/x_2 \\
Q\leqslant \min\{ x_3 Q^\prime ,x_1 R\} \\ 2\ell Q\leqslant Q^\prime < (2\ell+g)Q \\
Q^\prime > R+(2\ell-1)Q}} \frac{x_4}{Q^\prime(Q^\prime+x_4 Q)}, \\
\DD_2 (R) & = \sum\limits_{\substack{M\in\RR_O ,\, Q^\prime > R/x_2 \\
Q\leqslant \min\{ x_3 Q^\prime ,x_1 R\} \\ \lambda=Q^\prime/Q\in [2,2+g) ,\, Q^\prime >R+Q}} \frac{x_4}{Q^\prime(Q^\prime-x_4 Q)}, \\
\DD_3 (R) & = \sum\limits_{\substack{M\in\RR_O,\, Q^\prime > R/x_2 \\
Q\leqslant \min\{ x_3 Q^\prime ,x_1 R\} \\ \lambda=Q^\prime/Q \in \SS_3,\,\lambda > G^2 \\
Q(1+\{ \lambda\}) > R}} f_{x_4}^- (Q,Q^\prime) =
\sum_{\ell \geqslant 2} \sum\limits_{\substack{M\in\RR_O,\, Q^\prime > R/x_2 \\
Q\leqslant \min\{ x_3 Q^\prime ,x_1 R\} \\ 2\ell Q\leqslant Q^\prime < (2\ell+g)Q \\
Q^\prime > R+(2\ell -1)Q}} \frac{x_4}{Q^\prime(Q^\prime-x_4 Q)}.
\end{split}
\end{equation*}

Clearly $\DD_2(R)=0$ when $\min\{ x_1x_2,x_3\} \leqslant g^2$. When
$\min\{ x_1x_2,x_3\} > g^2$, the method employed in \eqref{4.1}--\eqref{4.3} leads, with $D_2$ as in \eqref{1.7}, to
\begin{equation*}
\DD_2 (R) =\frac{D_2(x_1,x_2,x_3,x_4)}{\zeta(2)}+O_\varepsilon \big( R^{-1+\varepsilon}\big).
\end{equation*}

The estimation of $\DD_1(R)$ is slightly more involved because $\ell$ can take infinitely many values.
Note that $\DD_1(R)=0$ unless $\min\{ x_1x_2,x_3\} >\frac{1}{2\ell+g}$.
For each $\ell\in\N$ consider the integral
\begin{equation*}
I_\ell^+ (R):= \iint\limits_{\substack{v\geqslant R/x_2,\, u\leqslant \min\{ x_3 v,x_1 R\} \\
2\ell u \leqslant v\leqslant (2\ell+g)u \\ v> R+(2\ell-1)u}}
\frac{x_4\, du\, dv}{v(v+x_4 u)} .
\end{equation*}
The change of variables $(v,u)=(Ry,Rx)$ shows that $I_\ell^+ (R)$ does not depend on $R$ and is given by \eqref{1.8}.
Note also that
\begin{equation}\label{4.4}
I_\ell^+ (R) \leqslant \int_{0}^{x_1} dx \int_{2\ell x}^{(2\ell+1)x} \frac{dy}{y^2} \ll \frac{1}{\ell^2} .
\end{equation}
A trivial estimate yields
\begin{equation*}
\begin{split}
\sum\limits_{\substack{\ell\geqslant R^{1/2} \\ R/x_2 \leqslant Q^\prime \leqslant (2\ell +1) R}} &
\sum\limits_{\frac{Q^\prime}{2\ell+1} \leqslant Q\leqslant \frac{Q^\prime}{2\ell}}
\frac{1}{Q^\prime(Q^\prime+x_4 Q)} \leqslant
\sum\limits_{\substack{\ell\geqslant R^{1/2} \\ 1\leqslant Q^\prime \leqslant (2\ell +1) R}}
\sum\limits_{\frac{Q^\prime}{2\ell+1} \leqslant Q\leqslant \frac{Q^\prime}{2\ell}}
\frac{1}{Q^2 \ell^2} \\ & \ll \sum_{\ell \geqslant R^{1/2}} \frac{1}{\ell^2} \sum_{Q\in [1,2R]}
\sum\limits_{Q^\prime \in [2\ell Q,(2\ell+1)Q]} \frac{1}{Q^2}\ll \frac{\log R}{R^{1/2}} ,
\end{split}
\end{equation*}
and thus in the definition of $\DD_1(R)$ we may take $\ell \in [1,R^{1/2}]$ inserting
an error term $\ll R^{-1/2}\log R$.
Employing Lemma \ref{Mob1}, we can express the resulting main term as
\begin{equation*}
\begin{split}
& \sum_{\ell\leqslant R^{1/2}} \hspace{-10pt} \sum\limits_{\substack{Q^\prime \geqslant R/x_2 \\
Q^\prime <(2\ell+g)x_1 R}} \hspace{-10pt} \Bigg( \frac{\varphi(Q^\prime)}{Q^\prime}
\int_{\frac{Q^\prime}{2\ell+g}}^{
\min\big\{ \frac{Q^\prime}{2\ell},x_3 Q^\prime,x_1 R,\frac{Q^\prime-R}{2\ell-1}\big\}}
\frac{x_4\, du}{Q^\prime(Q^\prime+x_4 u)} +O\big( Q^{\prime -2+\varepsilon}\big)\Bigg) \\
& =\Bigg( \sum_{\ell\leqslant R^{1/2}} \hspace{-10pt} \sum\limits_{\substack{Q^\prime \geqslant R/x_2 \\
Q^\prime < (2\ell +g)x_1 R}} \hspace{-23pt} \frac{\varphi(Q^\prime)}{Q^\prime} \hspace{-3pt}
\int_{\frac{Q^\prime}{2\ell+g}}^{\min\big\{ \frac{Q^\prime}{2\ell},x_3 Q^\prime,x_1 R,\frac{Q^\prime -R}{2\ell-1}\big\}}
\hspace{-10pt} \frac{x_4\, du}{Q^\prime (Q^\prime+x_4 u)} \Bigg) +O_\varepsilon \big( R^{-1/2+\varepsilon}\big).
\end{split}
\end{equation*}
By Lemma \ref{Mob2}, the main term above becomes
\begin{equation*}
\sum_{\ell\leqslant R^{1/2}} \Bigg( \frac{I_\ell^+}{\zeta(2)} +O \bigg( \frac{\log R}{R^{1/2}}\bigg) \Bigg),
\end{equation*}
and so
\begin{equation}\label{4.5}
\DD_1 (R) = \frac{1}{\zeta(2)}\sum_{\ell\leqslant R^{1/2}} I_\ell^+ +O_\varepsilon \big( R^{-1/2+\varepsilon}\big) .
\end{equation}
From \eqref{4.5} and \eqref{4.4} we eventually infer
\begin{equation*}
\DD_1 (R)=\frac{1}{\zeta(2)}\sum_{\ell\geqslant 1} I_\ell^+ +O_\varepsilon \big( R^{-1/2+\varepsilon}\big).
\end{equation*}
The sum $\DD_3 (R)$ is similarly estimated as in formulas \eqref{1.7} and \eqref{1.8}.

\section{Joint distribution for Nakada's $\alpha$-expansions}

We illustrate how explicit renewal type results can be obtained in the case of
Nakada's $\alpha$-expansions ${\rm NCF}_\alpha$, $\alpha \in [1/2,1]$. Such continued fractions, defined in \cite{Nak}, have been studied in \cite{Nak,Kra}.
Here the unit interval is replaced by $\Omega_\alpha =[\alpha-1,\alpha)$ and
the Gauss shift by the map $T_\alpha :\Omega_\alpha \rightarrow \Omega_\alpha$ defined for $x\neq 0$
by\footnote{Here we use the notation from Sections 5 and 6 of \cite{Kra}.}
\begin{equation*}
T_\alpha (x) =\left| \frac{1}{x} \right| -\left[ \left| \frac{1}{x}\right| +1-\alpha \right].
\end{equation*}
A construction of the natural extension $\overline{T}_\alpha$  on
a space $\underline{\Omega}_\alpha \subset \R^2$, together with an explicit invariant Borel probability measure $\mu_\alpha$ on $\underline{\Omega}_\alpha$
was found by Nakada \cite{Nak}. He also proved that
$(\underline{\Omega}_\alpha, \overline{T}_\alpha,\mu_\alpha)$ is a Kolmogorov automorphism.
With $g=1/G=1-g^2$ the set $\underline{\Omega}_\alpha$ is given for $g < \alpha \leqslant 1$ by
\begin{equation*}
[\alpha-1,(1-\alpha)/\alpha]\times [0,1/2) \cup
\big( (1-\alpha)/\alpha,\alpha\big) \times [0,1] \cup
[\alpha-1,0) \times \{ 1/2\} ,
\end{equation*}
and for $1/2 \leqslant \alpha \leqslant g$ by
\begin{equation*}
\begin{split}
& [\alpha-1,(1-2\alpha)/\alpha] \times [0,g^2) \cup
\big( (1-2\alpha)/\alpha, (2\alpha -1)/(1-\alpha)\big] \times [0,1/2) \\
& \qquad \cup \big( (2\alpha-1)/(1-\alpha),\alpha \big) \times [0,g) \cup [-g^2,(1-2\alpha)/\alpha ] \times \{ g^2\} \\
& \qquad \cup \big( (1-2\alpha)/\alpha,0\big) \times \{ 1/2\} .
\end{split}
\end{equation*}
Kraaikamp's thoughtful analysis (see especially Theorem (5.3) and Definitions (5.7) and (5.8) of \cite{Kra})
also provides characterizations of pairs of successive convergents for such continued fractions
if $\alpha\in [ 1/2,1]$.

\begin{prop}\label{Prop5.1}
For each $x\in\Omega_\alpha \setminus \Q$ the following are equivalent:
\begin{enumerate}
\item[(i)]
$P/Q,P^\prime /Q^\prime$ are successive convergents in $\operatorname{NCF}_\alpha (x)$ with $Q,Q^\prime >0$.
\item[(ii)]
$M=\left( \begin{smallmatrix} P & P^\prime \\
Q & Q^\prime \end{smallmatrix}\right) \in \operatorname{GL}_2(\Z)$ and $\big( E_M (x),1/\lambda_M \big) \in
\underline{\Omega}_\alpha$.
\end{enumerate}
\end{prop}

This dynamical system was studied by Kraaikamp \cite{Kra} in the more general setting of $S$-expansions, and the above proposition can be likewise
generalized if we replace ${\rm NCF}_\alpha (x)$ with ${\rm CF}_S(x)$, the $S$-expansion of $x$, and replace $\underline{\Omega}_\alpha$ with $\underline{\Omega}_S$,
the space of the natural extension associated to $S$.

We wish to estimate the Lebesgue measure $\LL^{(\alpha),\pm}_{x_1,x_2,x_3,x_4} (R)$ of the set of
numbers $x\in \Omega_\alpha\setminus\Q$ for which there exist successive convergents $P/Q$, $P^\prime / Q^\prime$ 
in $NCF_\alpha (x)$ that satisfy \eqref{1.1} and \eqref{1.2}.  We shall require that $x_1,x_2,x_3$ are in the set $(0,1]$ if $g<\alpha \leqslant 1$,
in $(0,1/2]$ if $\alpha=g$, and in $(0,g]$ if $1/2\leqslant \alpha < g$; moreover, we require $x_4 \in
(0,\alpha]$ when we look at $\LL^+$ and $x_4 \in ( 0,1-\alpha]$ when we look at $\LL^-$.
The set $\underline{\Omega}_\alpha$ is a union of rectangles and horizontal line segments, but we may ignore the line segments for large $R$: in particular,
the inequality $Q^\prime \geqslant R / x_2$ shows that the pair $(Q^\prime,Q)=(2,1)$ makes no contribution
to $\LL^\pm$ for $R >2$, so the situation $\lambda_M^{-1} =1/2$ can be ignored, and $\lambda_M$ is always rational, so the situation $\lambda_M^{-1}=g^2$ can also be ignored.
As a result, the cases that appear in $\LL^{(\alpha),\pm}_{x_1,x_2,x_3,x_4} (R)$ for $R>2$
are exactly:
\begin{equation*}
\begin{split}
\mbox{\rm For $g<\alpha \leqslant 1$:} & \quad  \begin{cases}
\lambda_M =Q^\prime /Q >2 \ \mbox{\rm and}\  \alpha -1 \leqslant E_M (x) <\alpha,\quad \mbox{\rm or}  \\
1\leqslant \lambda_M < 2 \  \mbox{\rm and} \  \frac{1-\alpha}{\alpha} < E_M(x) <\alpha.
\end{cases} \\
\mbox{\rm For $1/2 \leqslant \alpha \leqslant g$:} &  \quad
\begin{cases}
\lambda_M > G^2 \  \mbox{\rm and} \  \alpha-1\leqslant E_M(x) <\alpha ,\quad \mbox{\rm or} \\
2 < \lambda_M <G^2 \  \mbox{\rm and} \  \frac{1-2\alpha}{\alpha} < E_M (x) < \alpha,\quad \mbox{\rm or} \\
G < \lambda_M < 2 \  \mbox{\rm and} \  \frac{2\alpha-1}{1-\alpha} < E_M (x) < \alpha.
\end{cases}
\end{split}
\end{equation*}
The varying lower bounds on $\lambda_M$ depending on the value of $\alpha$ are the reason for our case-based restrictions on the values of $x_1, x_2, x_3$.

Let $\LL^+_{x_1,x_2,x_3,x_4} (\alpha;R)$ denote the Lebesgue measure of the set of
numbers $x\in [0,1]\setminus \Q$ for which there exists $M=\left( \begin{smallmatrix} P & P^\prime \\ Q & Q^\prime
\end{smallmatrix}\right) \in \operatorname{GL}_2(\Z)$ with $Q,Q^\prime >0$,
$P/Q,P^\prime / Q^\prime \in [\alpha -1,\alpha)$ and \eqref{1.1} together with $0\leqslant \frac{Q^\prime x-P^\prime}{-Qx+P} \leqslant x_4$ hold.
The corresponding set where the latter inequality is replaced by $-x_4 \leqslant \frac{Q^\prime x-P^\prime}{-Qx+P} \leqslant 0$
is denoted by $\LL^-_{x_1,x_2,x_3,x_4} (\alpha;R)$. In both cases, $x_1,x_2,x_3,x_4$ are parameters in $(0,1]$.
When $\alpha=1$, it is clear that $\LL^+$ is exactly the joint distribution considered in \cite{Ust} (where the notation used is $N(R)$).
However, by the following equation
\begin{equation*}
\begin{split}
\LL^\pm_{x_1,x_2,x_3,x_4} (\alpha;R) & = \sum_{Q^\prime \geqslant R/x_2}
\sum\limits_{\substack{Q\in ( 0,\min\{ x_3 Q^\prime,x_1 R\}] \\
P^\prime \in (\alpha-1)Q' +[0,Q') \\
P^\prime Q \equiv \pm 1 \hspace{-6pt} \pmod{Q^\prime} }} \frac{x_4}{Q^\prime (Q^\prime \pm x_4 Q)} \\
& = 2 \hspace{-7pt} \sum_{Q^\prime \geqslant R/x_2}
\sum\limits_{\substack{Q\in ( 0,\min\{ x_3 Q^\prime,x_1 R\}] \\ (Q,Q^\prime)=1}} \frac{x_4}{Q^\prime(Q^\prime \pm x_4 Q)} \\ & =
\LL^\pm_{x_1,x_2,x_3,x_4} (R) ,
\end{split}
\end{equation*}
we see that $\LL^\pm (\alpha;R)$ does not depend on $\alpha$.  As $R$ tends to infinity, $\LL^\pm$ converges to $2 F^\pm/\zeta(2)$.

The joint distributions $\LL^{(\alpha),\pm}$ and $\LL^\pm$ can now be directly related as below. For the sake of space and readability we omit the
appearance of $x_1$, $x_2$, and $R$, which are assumed to be the same on the left- and right-hand sides of the equations.

When $g<\alpha \leqslant 1$, we have
\begin{equation*}
\begin{split}
\LL^{(\alpha),+}_{x_3,x_4} & = \begin{cases}
\LL^{+}_{\min\{ x_3,1/2\},x_4} & \mbox{\rm if $0 \leqslant x_4 \leqslant (1-\alpha)/\alpha$,} \\
\LL^{+}_{x_3,x_4} -\LL^+_{x_3,(1-\alpha)/\alpha}
+\LL^+_{\min\{ x_3,1/2\},(1-\alpha)/\alpha} & \mbox{\rm if $(1-\alpha)/\alpha \leqslant x_4 <\alpha$,}
\end{cases} \\
\LL^{(\alpha),-}_{x_3,x_4} & = \LL^{-}_{\min\{x_3,1/2\},x_4} \quad \mbox{\rm if $0\leqslant x_4 \leqslant 1-\alpha.$}
\end{split}
\end{equation*}
When $1/2 \leqslant \alpha \leqslant g$, we have
\begin{equation*}
\begin{split}
\LL^{(\alpha),+}_{x_3,x_4}  & = \begin{cases}
\LL^+_{\min\{ x_3,1/2\},x_4}  & \mbox{\rm if $0\leqslant x_4 \leqslant (2\alpha-1)/(1-\alpha)$,} \\
\LL^{+}_{x_3,x_4} -\LL^{+}_{x_3,(2\alpha-1)/(1-\alpha)} +\LL^{+}_{\min\{ x_3,1/2\},(2\alpha-1)/(1-\alpha)} &
\mbox{\rm if $(2\alpha-1)/(1-\alpha) \leqslant x_4 < \alpha$.}
\end{cases}  \\
\LL^{(\alpha),-}_{x_3,x_4} & =\begin{cases} \LL^{-}_{\min\{ x_3,1/2\},x_4}  &
\mbox{\rm if $0\leqslant x_4 \leqslant (2\alpha-1)/\alpha,$} \\
\LL^{-}_{\min\{x_3,g^2\},x_4} +\LL^{-}_{\min\{ x_3,1/2\},(2\alpha-1)/\alpha} & \\
\qquad -\LL^{-}_{\min\{ x_3,g^2\},(2\alpha-1)/\alpha}
& \mbox{\rm if $(2\alpha-1)/\alpha \leqslant x_4 \leqslant 1-\alpha$.}
\end{cases}
\end{split}
\end{equation*}
Recall that $x_3 \leqslant g$ in this case.

\subsection*{Acknowledgements}
The second author acknowledges support from National Science Foundation grant
DMS 08-38434 ``EMSW21-MCTP: Research Experience for Graduate Students.''
We are grateful to Alexey Ustinov for suggesting that our method might also apply to NICF and for reference \cite{Ust2}
and to the referee for suggesting that our method would further apply to Nakada's $\alpha$-expansions.

\end{document}